\documentclass[11pt]{article}

\usepackage{ amssymb, amsmath, amsthm, amsfonts, mathrsfs, mathtools }
\usepackage[T1]{fontenc}
\usepackage{graphicx, latexsym, tabularx, shapepar, enumerate}
\usepackage[enableskew]{youngtab}
\usepackage{xy}
\xyoption{all}
\usepackage{tikz}
\usepackage{tikz-cd}
\usepackage{tkz-euclide}
\usepackage{fullpage}
\usepackage{stmaryrd}
\usepackage{enumerate}
\usepackage{float}
\usepackage{caption}
\usepackage[super]{nth}
\usepackage[backend=bibtex,style=numeric-comp, doi=false, url=false, isbn=false]{biblatex}
\bibliography{references} 

\usepackage{parskip}
\theoremstyle{plain}
\newtheorem{thm}{Theorem}[subsection]
\newtheorem{lemma}[thm]{Lemma}
\newtheorem{cor}[thm]{Corollary}
\newtheorem{prop}[thm]{Proposition}
\theoremstyle{definition}

\newtheorem{defn}[thm]{Definition}
\newtheorem{ex}[thm]{Example}

\newcommand{\CC}{\mathbb{C}}

\renewcommand{\phi}{\varphi}
\usepackage{ytableau}
\newcommand{\barr}[1]{\text{\footnotesize{\(\overline{\overline{#1}}\)}}}
\newcommand{\sbar}[1]{\text{\small{\(\overline{#1}\)}}}

\usetikzlibrary{shapes}
\tikzset{
    arrowMe/.style={
        postaction=decorate,
        decoration={
            markings,
            mark=at position 0.55 with {\arrow[thick,scale=1.5]{#1}}
        }
    }
}

\begin{document}

\newcommand{\Addresses}{{
  \bigskip

  \par Yulia Alexander, \textsc{Department of Mathematics and Computer Science, Wesleyan University, Middletown, CT 06459}\par\nopagebreak
  \textit{E-mail address}: \texttt{yalexandr@wesleyan.edu}

  \medskip

  \par Patricia Commins, \textsc{Department of Mathematics and Statistics, Carleton College, Northfield, MN 55057}\par\nopagebreak
  \textit{E-mail address}: \texttt{comminsp@carleton.edu}

  \medskip

  \par Alexandra Embry, \textsc{Department of Mathematics, Indiana University, Bloomington, IN 47405}\par\nopagebreak
  \textit{E-mail address}: \texttt{aiembry@iu.edu}
  
  \medskip
  
  \par Sylvia Frank, \textsc{Department of Mathematics and Statistics, Amherst College, Amherst, MA 01002}\par\nopagebreak
  \textit{E-mail address}: \texttt{sfrank20@amherst.edu}
  
  \medskip

  \par Yutong Li, \textsc{Department of Mathematics and Statistics, Haverford College, Haverford, PA 19041}\par\nopagebreak
  \textit{E-mail address}: \texttt{yli11@haverford.edu}
  
  \medskip
  
  \par Alexander Vetter, \textsc{Department of Mathematics and Statistics, Villanova University, Villanova, PA 19085}\par\nopagebreak
  \textit{E-mail address}: \texttt{avetter@villanova.edu}
  
}}

\title{Deformations of the Weyl Character Formula for $SO(2n+1,\mathbb{C})$ via Ice Models}
\author{Yulia Alexandr 
\and Patricia Commins \and Alexandra Embry \and Sylvia Frank \and Yutong Li \and Alexander Vetter}
\maketitle

\begin{abstract}
    We explore combinatorial formulas for deformations of highest weight characters of the odd orthogonal group $SO(2n+1)$. Our goal is to represent these deformations of characters as partition functions of statistical mechanical models -- in particular, two-dimensional solvable lattice models. In Cartan type $A$, Hamel and King \cite{hamel_symplectic_2002} and Brubaker, Bump, and Friedberg \cite{brubaker_schur_2011} gave square ice models on a rectangular lattice which produced such a deformation. Outside of type $A$, ice-type models were found using rectangular lattices with additional boundary conditions that split into two classes -- those with ``nested'' and ``non-nested bends.'' Our results fill a gap in the literature, providing the first such formulas for type $B$ with non-nested bends. In type $B$, there are many known combinatorial parameterizations of highest weight representation basis vectors as catalogued by Proctor \cite{proctor_young_1994}. We show that some of these permit ice-type models via appropriate bijections (those of Sundaram \cite{sundaram_orthogonal_1990} and Koike-Terada \cite{koike_young_1990}) while other examples due to Proctor do not.
\end{abstract}

\section{Introduction}
Statistical mechanics, and particularly solvable lattice models, have a long history as a source for interesting special functions related to combinatorics and representation theory. In this paper we focus on variants of the six-vertex, or square ice, models on finite two-dimensional lattices. We refer to a model as ``solvable'' if the underlying Boltzmann weights satisfy a Yang-Baxter equation which allows one to prove functional equations for the partition function and thus often ``solve'' for it in closed form. For example, Kuperberg~\cite{kuperberg_symmetry_2002}, building on the work of Izergin and Korepin~\cite{Izergin, 0305-4470-25-16-010}, used ice models to enumerate symmetry classes of alternating sign matrices (ASMs) . Hamel and King \cite{hamel_symplectic_2002} showed that a deformation of Schur polynomials due to Tokuyama could be represented as a partition function of an ice model and Brubaker, Bump, and Friedberg \cite{brubaker_schur_2011} showed that the model is solvable and gave new proofs of Tokuyama's identities.

These latter results opened up the possibility of using solvable lattice models to represent deformations of highest weight characters for all Cartan types. Perhaps not surprisingly, many of Kuperberg's ice models used to find symmetry classes of ASMs play a role here as well; roughly, ASMs generalize permutation matrices and their symmetry classes reflect embeddings of classical Weyl groups in a large enough general linear group. The models for other types split into two main classes according to the shape of the boundary of the lattice model. We refer to them as the ``nested'' and ``non-nested'' models, as pictured in Figure \ref{nested-type-C} and \ref{non-nested-type-C}, respectively.

\begin{figure}[H]
\minipage[b]{.5\textwidth}
\begin{center}\begin{tikzpicture}[scale=0.8]
%lattice
\foreach \x in {1,2}{
	\draw[thick] (\x, 0.3) -- (\x, 4.7);}
\foreach \y in {1,2,3,4}{
	\draw[thick] (0,\y) -- (2.5,\y);}
%bends
\foreach \top in {3}{
	\draw[thick] (2.5, \top) to [out=0, in=90] (3, \top - 0.5) to [out=-90, in=0](2.5, \top - 1);
	\fill (3, \top - 0.5) circle (2pt);}
\draw [thick] (2.5,4) arc (90:0:1.5);
\draw [thick] (4,2.5) arc (0:-90:1.5);
\fill (4,2.5) circle (2pt);
%labels
\node [label=left:$x_i$] at (0,1) {};
\node [label=left:$x_j$] at (0,2) {};
\node [label=left:$\overline{x_j}$] at (0,3) {};
\node [label=left:$\overline{x_i}$] at (0,4) {};
\end{tikzpicture}
\end{center}
\caption{Nested ice for type C}\label{nested-type-C}
\endminipage
\minipage[b]{0.5\textwidth}
	\begin{center}
\begin{tikzpicture}[scale=0.8]
%lattice
\foreach \x in {1,2}{
	\draw[thick] (\x, 0.3) -- (\x, 4.7);}
\foreach \y in {1,2,3,4}{
	\draw[thick] (0,\y) -- (2.5,\y);}
%bends
\foreach \top in {4,2}{
	\draw[thick] (2.5, \top) to [out=0, in=90] (3, \top - 0.5) to [out=-90, in=0](2.5, \top - 1);
	\fill (3, \top - 0.5) circle (2pt);}
%label
\node [label=left:$x_i$] at (0,1) {};
\node [label=left:$\overline{x_i}$] at (0,2) {};
\node [label=left:$x_j$] at (0,3) {};
\node [label=left:$\overline{x_j}$] at (0,4) {};
	\end{tikzpicture}
	\end{center}
	\caption{Non-nested ice for type C}\label{non-nested-type-C}\endminipage
\end{figure}

For nested models, Brubaker and Schultz \cite{brubaker_six-vertex_2015} gave deformations of the Weyl character formula for all classical types which generalized deformations of the denominator formula due to Okada \cite{Okada1993}. However, they were not able to evaluate the deformed character explicitly (say, by a determinantal or alternator expression). Moreover, in type $A$, admissible states of ice are in bijection with strict Gelfand-Tsetlin patterns which explains their connection to representation theory. However the nested models are not known to arise from a natural Gelfand-Tsetlin-like parametrization of basis vectors (or any natural functorial operation like a branching rule) in the associated highest weight representations.

In this paper, we focus solely on non-nested models. Non-nested models have been shown to be in bijection with Gelfand-Tsetlin patterns for other types. In type $C$ Hamel and King \cite{hamel_symplectic_2002} gave a model, which was shown to be solvable by Ivanov, in bijection with strict patterns for symplectic representations due to Zhelobenko \cite{zhelobenko_classical_1962}. Formulas for other types remained open. This paper explores non-nested models in type $B$. Here the Gelfand-Tsetlin type patterns for highest weight representations are a Wild West of possible varieties, as catalogued by Proctor, not only due to the existence of spin representations but also multiple constructions of patterns. Various known constructions of patterns in type $B$ are due to Sundaram \cite{sundaram_tableaux_1990}, Koike-Terada \cite{koike_young_1990}, and Proctor \cite{proctor_young_1994}.

Using tableaux rules from Sundaram, we create new Gelfand-Tsetlin-like patterns. We then take the strict Gelfand-Tsetlin-like patterns and create a bijection with new ice models as in Figure \ref{SICE}.

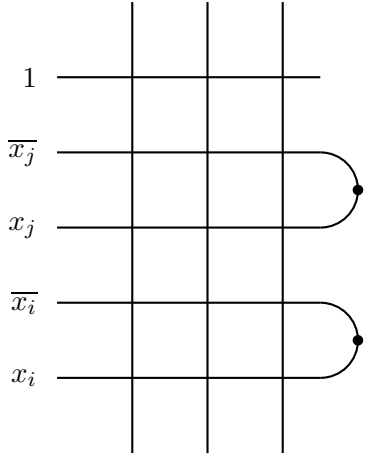
\begin{figure}
\centering
\begin{tikzpicture}
%lattice
\foreach \x in {0,1,2}{
	\draw[thick] (\x, 0) -- (\x, 6);}
\foreach \y in {1,2,3,4,5}{
	\draw[thick] (-1,\y) -- (2.5,\y);}
%bends
\foreach \top in {4,2}{
	\draw[thick] (2.5, \top) to [out=0, in=90] (3, \top - 0.5) to [out=-90, in=0](2.5, \top - 1);
	\fill (3, \top - 0.5) circle (2pt);}
% row labels
\node [label=left:$1$] at (-1,5) {};
\node [label=left:$x_i$] at (-1,1) {};
\node [label=left:$\overline{x_i}$] at (-1,2) {};
\node [label=left:$x_j$] at (-1,3) {};
\node [label=left:$\overline{x_j}$] at (-1,4) {};

\end{tikzpicture}\caption{Non-nested ice for type B}\label{SICE}
\end{figure}

By assigning Boltzmann weights to the vertices in the resulting ice model, we get a partition function by summing over all admissible states. The result is equal to the product of the type B deformation and the Weyl character formula for $SO(2n+1,\CC)$. In particular,\\
\begin{thm}
Let $\lambda$ be a partition that indexes an irreducible representation of $SO(2n+1,\mathbb{C})$, $s_{\lambda}^{so}$ the corresponding character. Then, the partition function, $\mathcal{Z}(\lambda)$, on ice models stemming from Sundaram tableaux is equal to:
\begin{equation}\label{BS-deformation}
    \mathcal{Z}(\lambda)=\prod_{i=1}^{n}(1+tz_{i})\prod_{i<j}(1+tz_{i}z_{j})(1+z_{i}z_{j}^{-1})s_{\lambda}^{so}
\end{equation}
\end{thm}

We further explore the possibility of using ice models for the Koike-Terada and Proctor patterns, showing that it is increasingly difficult to find a solvable model in these cases (Sections \ref{KTICEsection} and \ref{ProctorICE}.

Finally, we mention some additional applications. First, ice models seem particularly appealing because they allow one to simultaneously observe branching rules (via connections to patterns) and Weyl group symmetry (via the Yang-Baxter equation). Seeing the Weyl group action on patterns is more complicated \cite{BerensteinKirillov}. Second, the non-nested deformations match the Shintani-Casselman-Shalika \cite{CasselmanShalika} formula for the special values of spherical Whittaker functions on $p$-adic reductive groups. These are an important component in the local theory of automorphic forms and their connection to solvable lattice models hints at potentially deep connections between automorphic forms and quantum groups. Our results may form the basis for further generalizations to Whittaker functions on metaplectic covers of symplectic groups (whose dual group is of type $B$), an area under active development in type $A$ already \cite{BBB, BBBG}.

\subsection{Content of Paper}
This paper is organized as follows. We begin by introducing combinatorial objects in Section \ref{Prelim} including Gelfand-Tsetlin Patterns, semistandard Young tableaux, shifted Young tableaux, and ice model, and most importantly, Tokuyama's formula, which gives us the notion of a deformed character. In Section \ref{SundaramSection}, we use tableaux due to Sundaram to create Gelfand-Tsetlin-like patterns, shifted tableaux, and strict Gelfand-Tsetlin-like patterns. In Section \ref{TableauxIceModelsSundaram} introduces the ice model corresponding to this type of shifted tableaux. In Section \ref{MainTheorem}, we assign Boltzmann weights to the ice model and present our main results, Theorem \ref{CdeformationtimesBcharacter} and Corollary \ref{MainCor}, which give a deformation of the Weyl Character Formula for $SO(2n+1,\mathbb{C})$ via ice models. In Section \ref{KTSection}, we discuss tableaux due to Koike and Terada, introducing another class of Gelfand-Tsetlin-type patterns as well as ice models. Finally, in Section \ref{ProctorSection} we examine tableaux due to Proctor and show that these branching rules cannot be realized through ice.

\subsection{Acknowledgements} 
This research was conducted at the 2018 University of Minnesota-Twin Cities REU in Algebraic Combinatorics. Our research was supported by NSF RTG grant DMS-1745638. We thank Dr. Benjamin Brubaker and Katherine Weber for their advice and support during our time in the Twin Cities. We also thank Jiyang Gao for his contributions to the project.

\section{Combinatorial Representation Theory for $GL(n, \CC)$}\label{Prelim}

\subsection{Tokuyama's Formula}
In the well-studied case of $GL(n,\mathbb{C})$, Tokuyama proposed a generalization that recovers a deformation of Weyl's formula by summing over a combinatorial basis, namely, the strict Gelfand-Tsetlin patterns. Throughout, let $\lambda=(\lambda_{n}\geq\lambda_{n-1}\geq\cdots\geq\lambda_{1}\geq 0)$ be a partition with at most $n$ parts, which indexes an irreducible representation of $GL(n, \CC)$.

\par A Gelfand-Tsetlin (GT) pattern is a triangular array of non-negative integers:
\begin{equation*}
\begin{array}{ccccccccc}
a_{1,1}& & a_{1,2}  & & \cdots & & a_{1,n-1} & & a_{1,n}\\
&a_{2,1} &&a_{2,2}  & \cdots & a_{2,n-1} && a_{2,n}\\
&&&& \cdots &&&& \\
&&& a_{n-1,1} && a_{n-1,2}&&& \\
&&&& a_{n,1} &&&&
\end{array}
\end{equation*}
subject to:
\begin{enumerate}
	\item $a_{i,j-1}\geq a_{i,j}\geq a_{i,j+1}\geq 0$ 
	\item $a_{i-1,j}\geq a_{i,j}\geq a_{i-1,j+1}$
\end{enumerate}
For any triplet $\begin{array}{ccc}a&&b\\&c&\end{array}$ in a GT pattern, we say the entry $c$ is \textit{special} if $a > c > b$, and \textit{left-leaning} if $a = c$.
A strict Gelfand-Tsetlin pattern is such an array where the rows are strictly decreasing. 

\par Alternatively, the irreducible representation indexed by $\lambda$ can be parametrized by semistandard Young tableaux of shape $\lambda$, filled with the alphabet $\{1<2< \dots <n\}$ such that:
\begin{enumerate}
	\item Rows are weakly increasing.
	\item Columns are strictly increasing.
\end{enumerate}

\par Let $SSYT(\lambda)$ be the set of semistandard Young tableaux with shape $\lambda$. The Schur polynomial can be expressed as a sum over column-strict Young tableaux. In particular, the following is well-known and can be found in Chapter 7 of \cite{Stanley}.\\
\begin{thm}
Let $\mathbf{z}^{wt(T)} = \prod_{i=1}^{n}z_{i}^{c(i)}$, and $c(i)$ is the number of i's in a tableau $T$. Then,
\begin{equation*}
s_{\lambda}(\mathbf{z})=\sum_{T\in SSYT}\mathbf{z}^{wt(T)}
\end{equation*}
\end{thm}
\par It is well known that we have a bijection between the two combinatorial bases for the characters of irreducible representations, as illustrated in Example \ref{tab-GT-bijeciton}.
\begin{equation*}
SSYT(\lambda)\xleftrightarrow{1 \text{ to }1} GT(\lambda)
\end{equation*}

\begin{ex}\label{tab-GT-bijeciton}
\begin{equation*}
\ytableausetup{boxsize=2em}
\begin{ytableau}
    1 & 1 & 1 & 3 & 3\\
    2 & 2 & 2 \\
    3 & 3 
\end{ytableau}
\quad \leftrightarrow
\begin{array}{cccccc}
& 5 && 3 && 2 \\
&& 3 && 3 & \\
&&& 3 &&
\end{array}
\end{equation*}
\end{ex}
\par Further, it follows that, if $GT(\lambda)$ is the set of Gelfand-Tsetlin patterns with top row $\lambda$, and $R_{i}$ be the sum of the $i^{th}$ row in the GT pattern. The Schur polynomial associated to $\lambda$ can be expressed as:
\begin{equation*}
s_{\lambda}(\mathbf{z})=\sum_{P\in GT(\lambda)}\mathbf{z}^{wt(P)},
\end{equation*}
where $\mathbf{z}^{wt(P)}:=z_{1}^{p_{1}}z_{2}^{p_{2}}\dots z_{n}^{p_{n}}$, and $wt(P)$ is the $n$-tuple $(p_{1},p_{2},\dots ,p_{n})$, with $p_{i}=R_{i}-R_{i+1}$ for $i\in\{1,\dots,n-1\}$,  $p_{n}=R_{n}$.
\par Motivated by the above bijection, we present the following deformation in type A:\\ 
\begin{thm}[Tokuyama's Formula, Theorem 2.1~\cite{tokuyama_generating_1988}]
Let $t$ be an arbitrary parameter, $\lambda$ any partition, $\rho=(n-1,n-2,\dots,1,0)$, and $SGT(\lambda+\rho)$ the set of strict Gelfand-Tsetlin patterns with top row $\lambda+\rho$. Then,
\begin{equation}\label{TokuyamasFormula}
\prod_{1\leq i<j\leq n} (z_{i}+tz_{j})s_{\lambda}(\mathbf{z})=\sum_{T\in SGT(\lambda+\rho)} (1+t)^{S(T)}t^{L(T)}\mathbf{z}^{wt(T)}
\end{equation}
where 
\begin{itemize}
    \item $S(T)$ is the number of special entries in $T$
    \item $L(T)$ is the number of left-leaning entries in $T$
\end{itemize}
\end{thm}

\par Under the specialization of $t=-1$, the formula yields a deformation of the Weyl character formula for $GL(n,\mathbb{C})$. When $t=0$, Tokuyama's formula gives the original description of the character in terms of Gelfand-Tsetlin patterns.

\par Just like the bijection between GT patterns and seminstandard Young tableaux, shifted tableaux are in bijective correspondence with strict Gelfand-Tsetlin patterns. Let $\rho=(n-1,n-2,\dots,0)$. We consider a shifted Young tableaux to be of shape $\lambda + \rho$ by attaching a standard tableaux of shape $\rho$ to the left side of the standard tableaux of shape $\lambda$. The example below illustrates this process.\\
\begin{ex}
Let $\lambda=(5,3,2,1)$, so $\rho=(3,2,1,0)$. Then, a shifted tableaux of shape $\lambda+\rho$ looks like:
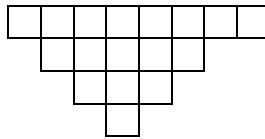
\begin{equation*}
\young(\:\:\:\:\:\:\:\:,:\:\:\:\:\:,::\:\:\:,:::\:)
\end{equation*}
From this change, the rules for filling the diagram are modified accordingly:
\begin{enumerate}
	\item Rows are weakly increasing.
	\item Columns are weakly increasing.
	\item Diagonals are strictly increasing.
\end{enumerate}
Let $SYT(\lambda +\rho)$ be the set of shifted Young tableaux satisfying the above conditions. From these tableaux, we again have a bijection: 
\begin{equation*}
SYT(\lambda+\rho)\xleftrightarrow{1 \text{ to }1} SGT(\lambda+\rho)
\end{equation*}
\end{ex}

\subsection{Ice Models}
In the case of $GL(n, \mathbb{C})$, states of ice are certain tetravalent directed graphs. More precisely, we require two edges pointed to and from each vertex. Let $\lambda=(\lambda_{n}\geq\lambda_{n-1}\geq\dots\geq\lambda_{1}\geq 0)$ be a partition, then the admissible ice states corresponding to $(\lambda+\rho)$ are $n\times (n+\lambda_n)$ grid graphs satisfying:
\begin{enumerate}
	\item The rows are labeled from bottom to top in increasing order of the alphabet for $GL(n,\mathbb{C})$, and columns are numbered $0$ through $\lambda_{n}+n-1$ as we move from right to left.
	\item All arrows on the left-hand side point to the right.
	\item All arrows on the bottom point down.
	\item All arrows on the right-hand side point to the left.
	\item If $i\in (\lambda+\rho)$, the top-most arrow in column $i$ points up; otherwise it points down.
\end{enumerate}
The final four conditions are boundary conditions, which we show in Figure \ref{IceBoundary}.
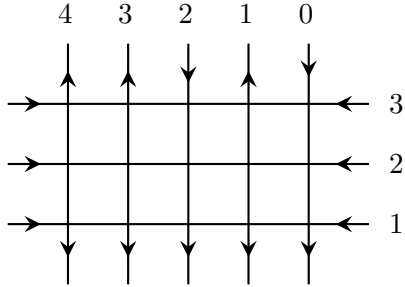
\begin{figure}[H]\centering 
\begin{tikzpicture}[scale=.8]
%lattice
\foreach \x in {0,1,2,3,4}{
	\draw[thick] (\x, -1) -- (\x, 3);}
\foreach \y in {0,1,2}{
	\draw[thick] (-1,\y) -- (5,\y);}

%down arrows
\foreach \x/\y in {0/0,1/0,2/0, 3/0, 4/0, 2/2.9, 4/3}{
	\draw[arrowMe=stealth] (\x, \y) -- (\x, \y-1);}
%right arrows
\foreach \x/\y in {-1/0,-1/1,-1/2}{
	\draw[arrowMe=stealth] (\x, \y) -- (\x+1, \y);}
%up arrows
\foreach \x/\y in {0/2,1/2, 3/2}{
	\draw[arrowMe=stealth] (\x, \y) -- (\x, \y+1);}
	%left arrows
\foreach \x/\y in {5/0,5/1,5/2} {
    \draw[arrowMe=stealth] (\x, \y) --(\x-1, \y);}

		\node [label=right:$3$] at (5, 2) {};
	\node [label=right:$2$] at (5, 1) {};
	\node [label=right:$1$] at (5, 0) {};
	\node [label=right:$4$] at (-0.5, 3.5) {};
	\node [label=right:$3$] at (0.5, 3.5) {};
	\node [label=right:$2$] at (1.5, 3.5) {};
	\node [label=right:$1$] at (2.5, 3.5) {};
	\node [label=right:$0$] at (3.5, 3.5) {};
\end{tikzpicture}
    \caption{Ice Boundary Conditions for $\lambda+\rho = (4, 2, 1)$ }
\label{IceBoundary}\end{figure}

The remaining edges can be filled such that every vertex has two arrows pointing in and two arrows pointing out. It turns out the admissible states of the ice models are in bijection with strict Gelfand-Tsetlin patterns. That is, 
\begin{equation*}
SGT(\lambda+\rho)\xleftrightarrow{1 \text{ to }1} ICE(\lambda+\rho)
\end{equation*}
We can view the vertices of the ice model as points on a lattice. For $T\in SGT$, and $a_{i,j}\in T$, then the vertex in the $i^{th}$ row from the top and in the column indexed by $a_{i,j}$ has an up arrow directly above it. The remaining column edges have down arrows.

\subsubsection{Boltzmann Weights and Partition Function}
Depending on the orientation of the arrows around a vertex, each vertex is assigned a particular Boltzmann weight. Invalid configurations of arrows have a Boltzmann weight of 0. Let $i$ indicate the $i^{th}$ row of the ice model. The following are nonzero Boltzmann weights for the $GL(n,\mathbb{C})$ ice models as provided in \cite{brubaker_schur_2011}.
\begin{figure}[H]
\begin{center}
\begin{tabular}
{|c|c|c|c|c|c|}
\hline
$\begin{tikzpicture}
\node [label=left:$j$] at (0.5,1) {};
\node [label=right:$ $] at (1.5,1) {};
\node [label=above:$ $] at (1,1.5) {};
\node [label=below:$ $] at (1,0.5) {};
\draw [>->] (0.5,1) -- (1.5,1);
\draw [>->] (1,1.5) -- (1,0.5);
\end{tikzpicture}$
&
$\begin{tikzpicture}
\node [label=left:$j$] at (0.5,1) {};
\node [label=right:$ $] at (1.5,1) {};
\node [label=above:$ $] at (1,1.5) {};
\node [label=below:$ $] at (1,0.5) {};
\draw [<-<] (0.5,1) -- (1.5,1);
\draw [<-<] (1,1.5) -- (1,0.5);
\end{tikzpicture}$
&
$\begin{tikzpicture}
\node [label=left:$j$] at (0.5,1) {};
\node [label=right:$ $] at (1.5,1) {};
\node [label=above:$ $] at (1,1.5) {};
\node [label=below:$ $] at (1,0.5) {};
\draw [>->] (0.5,1) -- (1.5,1);
\draw [<-<] (1,1.5) -- (1,0.5);
\end{tikzpicture}$
&
$\begin{tikzpicture}
\node [label=left:$j$] at (0.5,1) {};
\node [label=right:$ $] at (1.5,1) {};
\node [label=above:$ $] at (1,1.5) {};
\node [label=below:$ $] at (1,0.5) {};
\draw [<-<] (0.5,1) -- (1.5,1);
\draw [>->] (1,1.5) -- (1,0.5);
\end{tikzpicture}$
&
$\begin{tikzpicture}
\node [label=left:$j$] at (0.5,1) {};
\node [label=right:$ $] at (1.5,1) {};
\node [label=above:$ $] at (1,1.5) {};
\node [label=below:$ $] at (1,0.5) {};
\draw [<->] (0.5,1) -- (1.5,1);
\draw [>-<] (1,1.5) -- (1,0.5);
\end{tikzpicture}$
&
$\begin{tikzpicture}
\node [label=left:$j$] at (0.5,1) {};
\node [label=right:$ $] at (1.5,1) {};
\node [label=above:$ $] at (1,1.5) {};
\node [label=below:$ $] at (1,0.5) {};
\draw [>-<] (0.5,1) -- (1.5,1);
\draw [<->] (1,1.5) -- (1,0.5);
\end{tikzpicture}$
\\
%&&&&&\\ 
\hline
$NW=1$&$SE=z_i$&$SW=t$&$NE=z_i$&$NS=z_i(t+1)$&$EW=1$\\ \hline
\end{tabular}
\caption{Boltzmann weights for each of the six valid vertex orientations}
\label{BoltzmannWeights}
\end{center}
\end{figure}

From this description of ice models for $GL(n,\mathbb{C})$, we define a partition function $Z(\lambda)$. Let $ICE(\lambda+\rho)$ be the set of all admissible ice states with top row $\lambda +\rho$. For a particular $\mathcal{I}\in ICE(\lambda+\rho)$, assign the Boltzmann weights as in Figure \ref{BoltzmannWeights}. Let $w(\mathcal{I})$ be the product of the Boltzmann weights of each vertex in the ice model. Then, our partition function is the following:
\begin{equation*}
Z(\lambda)=\sum_{\mathcal{I}\in ICE}w(\mathcal{I})
\end{equation*}
The following is proved in \cite{brubaker_schur_2011}:
\\ \begin{thm}
\begin{equation*}
Z(\lambda)= \prod_{1\leq i<j\leq n} (z_{i}+tz_{j})s_{\lambda}(\mathbf{z})
\end{equation*}
\end{thm}
\par Notice that the partition function is again an equivalent description of Tokuyama's formula, hence the deformed character of $GL(n,\CC)$.

\section{From Sundaram Tableaux to Ice Models}\label{SundaramSection}
\par In this section, we explore known tableaux for $SO(2n+1,\mathbb{C})$ due to Sundaram \cite{sundaram_orthogonal_1990}. We discuss the origins of these tableaux rules through branching rules. We then create Gelfand-Tsetline-like patterns in bijection with these tableaux. Next, from the tableaux, we create shifted tableaux, and from the Gelfand-Tsetlin-like patterns, we create strict Gelfand-Tsetlin-like patterns. We show the shifted tableaux and strict Gelfand-Tsetline-like patterns are in bijection. Finally, we create new ice models that are in bijection with these shifted tableaux and strict Gelfand-Tsetlin-like patterns are in bijection.
\subsection{Construction}
Let $\lambda = (\lambda_{n}\geq\lambda_{n - 1}\geq\dots\lambda_{2}\geq 0)$ be a partition with n parts. Then, $\lambda$ is the highest weight of an irreducible representation of $SO(2n+1)$. Further, we can find the character of this representation through the use of certain Young tableaux. Let $\lambda$ define the length of the rows of such a tableaux. We fill the Young diagram with the alphabet $\{1<\bar1<\dots<n<\bar n<0\}$ and the following rules:
\begin{enumerate}
\item Rows are weakly increasing.
\item Columns weakly increase, with strictly increasing nonzero entries.
\item No row contains multiple 0s.
\item In row i, all entries are greater than or equal to i.
\end{enumerate}
We define the weight of a particular tableaux as the following:
\begin{equation*}
\prod_{i=1}^{n}x^{w(i)-w(\bar i)}
\end{equation*}
where $w(i)$ and $w(\bar i)$ are the number of $i$ and $\bar i$ in the Young tableaux. Sundaram shows that
\begin{equation*}
s_{(\lambda)}^{so}=\sum_{T\in SSYT}\prod_{i=1}^{n}x^{w(i)-w(\bar i)}
\end{equation*}
where SSYT is the set of Young diagrams filled according to the above rules, and $s_{(\lambda)}^{so}$ is the character of the irreducible representation of $SO(2n+1)$ indexed by $\lambda$.
\subsection{Branching Rules}
These tableaux rules derive from branching rules of the irreducible representations of $SO(2n+1, \CC)$, which we shall explore in the following section.
\subsubsection{Relation to Sp(2n) Tableaux}
Sundaram's tableaux for $SO(2n+1,\CC)$ is a result of the relationship:
\begin{equation}\label{HorizontalStripEquation}
s_{\lambda}^{so} = \sum_{\mu\subseteq\lambda}s_{\mu}^{sp}
\end{equation}
where $\mu = (\mu_{1}\geq\mu_{2}\geq\dots\mu_{n-1}\geq\mu_{n}>0)$, $s_{\mu}^{sp}$ is the character of the irreducible representation of $Sp(2n)$ indexed by $\mu$, and $\mu\subseteq\lambda$ if and only if $\lambda_{i}-\mu_{i}\leq 1$ for all $i$ (they differ by a horizontal strip). Further, we can construct tableaux for $Sp(2n)$ using the exact same rules as the $SO(2n+1,\mathbb{C})$ tableaux, excluding 0 in the alphabet, as described by King \cite{hamel_symplectic_2002}. Thus, it remains to understand the motivation for the tableaux rules for $Sp(2n)$.
\subsubsection{Restriction of Sp(2n) Representations}
In 1962, Zhelobenko showed the restriction $Sp(2n)\downarrow Sp(2n-2)\times U(1)$ is as follows \cite{zhelobenko_classical_1962}:
\begin{equation*}
(\lambda)\downarrow=\sum_{x,y}(\lambda/x\cdot y)\{x-y\}
\end{equation*}
where x and y are one-partitions, and / and $\cdot$ are Schur function quotients and products determined by the Littlewood-Richardson rule, which is found in \cite{classical_lie_groups}. For our purposes $\lambda /z$ means removing $z$ boxes, no two from the same column. Further, $x$ means the number of removals of $n$ and $y$ refers to the number of removals of $\bar n$. The tableaux construction from King arises because the formula above implies that the $n^{th}$ row in $T_{\lambda}$ has a value other than $n$ or $\bar n$, then that tableaux vanishes in the restriction. Hence, the fourth condition in our tableaux rules is referred to as the symplectic condition on page 3160 in \cite{classical_lie_groups}. From that restriction relation, Wallach and Yacobi gave an explicit branching formula. Further, Howe, Lavicka, Lee, and Soucek describe this in terms of Young diagrams, and use it in order to produce a Skew Pieri rule for $Sp(2n)$ \cite{howe_roger}. 

\subsection{Gelfand-Tsetlin-Type Patterns}\label{Sundaram-GT}
We define a new type of Gelfand-Tsetlin pattern that we will use to construct ice models. The patterns have the following shape:
\begin{equation*}
\begin{array}{ccccccccccccc}
& a_{1,1} && a_{1,2} && \cdots && a_{1,n} && 0&&&  \\
&& a_{2,1} && a_{2,2} && \cdots && a_{2,n}&&&& \\
&&& a_{3,1} && a_{3,2} && \cdots && a_{3,n} &&& \\
&&&&&&\cdots&&&&&&\\ 
&&&&&&a_{2n-2,1} && a_{2n-2,2}&&&&\\
&&&&&&& a_{2n-1,1}&& a_{2n-1,2}&&&\\
&&&&&&&& a_{2n, 1} &&&&\\
&&&&&&&&& a_{2n+1, 1} &&&
\end{array}
\end{equation*}
and they following the same interleaving rules as regular Gelfand-Tsetlin patterns.
\subsection{Shifted Tableaux and Strict Gelfand-Tsetlin-Type Patterns}\label{shiftedStrictSundaram}
\par In this case, our shifted tableaux rules are the following:
\begin{enumerate}
	\item Rows are weakly increasing.
	\item Columns are weakly increasing.
	\item Diagonals are strictly increasing.
	\item No row contains consecutive 0s.
	\item The first entry in row $i$ is either $i$ or $\bar i$.
\end{enumerate}

\par For us, we only want strict Gelfand-Tsetlin-type patterns in order to construct valid corresponding ice models. In the Sundaram case, in order to create a bijection with the shifted tableaux, our "strict" Gelfand-Tsetlin-type patterns will include the following constraints: 
\begin{enumerate}
    \item $a_{i,j-1}>a_{i,j}>a_{i,j+1}\geq 0$
	\item $a_{2k,n-k+1}\neq 0$
	\item $a_{1,k}-a_{2,k}\leq 1$
\end{enumerate}
The first condition simply means that each row in our pattern is strictly increasing, which allows us to create an ice model out of these patterns. The second condition is precisely equivalent to condition 5 of our shifted tableaux. The third condition is equivalent to the fourth condition of our shifted tableaux. Notice that the top row of our pattern is exactly equivalent to $\lambda + \rho$ and we add a 0 at the end for convenience.\\
\begin{lemma}
For a given partition $\lambda$, we form the shifted tableaux of shape $\lambda+\rho$, and we let our top row of our Gelfand-Tsetlin pattern be equivalent to $\lambda + \rho$ with a 0 at the end. Then, the permissible fillings of the shifted tableaux are in bijection with the Gelfand-Tsetlin-like patterns.
\end{lemma}
\begin{proof}
We propose a reversible algorithm for this process. \\
\par First, let row $2n+1$ of our Gelfand-Tsetlin type pattern correspond to filling the tableaux with 1, let row 2$n$ correspond to filling with $\bar 1$, and in general, let row $2k+1$ correspond to n-k+1 and let row 2$k$ correspond to $\overline{n-k+1}$ for $k\geq 1$. Finally, let row $1$ of the pattern correspond to filling with 0. \\
\par Beginning from the bottom row of the Gelfand-Tsetlin type pattern, we fill out the shifted tableaux using the numbers just described, as done in the $GL(n,\mathbb{C})$ case. It is easy to see that Constraint 2 of the strict Gelfand-Tsetlin type patterns implies that either $i$ or $\bar i$ will be on the main diagonal of row $i$ in the tableaux. This algorithm also clearly forces our rows to be weakly increasing. Further, because the rows of the Gelfand-Tsetlin pattern are strictly decreasing, our corresponding columns will be weakly increasing. Further, condition 4 ensures that we do not add two 0s to any row in the shifted tableaux. Additionally, all diagonals are strictly increasing increasing because of our strict decreasing conditions in the GT model.\\
\par We claim that this algorithm is reversible. In particular, we can form a GT pattern from the bottom up by looking at the filled shifted tableaux. The conditions are clear from this process. It is also clear this algorithm produces a unique result in either direction. Thus, we have a bijection.
\end{proof}

\subsection{Sundaram Tableaux and Ice Models}\label{TableauxIceModelsSundaram}
We now consider ice models of a different shape. In particular if $\lambda=(\lambda_{n}>\cdots>\lambda_{1}\geq 1)$ is strict, then we have $\lambda_{n}$ columns and $2n+1$ rows. We also have the following conditions:
\begin{enumerate}
    \item The rightmost arrow of the first row always points out.
    \item Further, the rows labeled with $i$ and $\bar i$ are connected by a loop on the right hand side. They are either connected by an A or B loop. 
\end{enumerate}

\begin{center}
\begin{tikzpicture}[scale=1]
\node[anchor = west, inner sep = 3pt] at (4.5, 1) {$A$};
\node[anchor = west, inner sep = 3pt]  at (8.5,1) {$B$};
\draw[thick, arrowMe=stealth] (4.5,1) to [out=-90, in=0] (4, 0.5);
\draw[thick, arrowMe=stealth](4, 1.5) to [out=0, in=90](4.5,1);
\draw[thick, arrowMe=stealth] (8.5,1) to [out=90, in=0] (8,1.5);
\draw[thick, arrowMe=stealth] (8,.5) to [out=0, in=-90] (8.5,1);
\foreach \nd in {4.5, 8.5}{
	\fill (\nd, 1) circle (2pt);}
\end{tikzpicture}
\end{center}

\begin{ex}
\par Let $\lambda=(3,2)$, and consider the following Gefland-Tsetlin-Type pattern:
\begin{figure}[H]\centering
\begin{equation*}
\begin{array}{ccccccccccccc}
& 3 && 2 && 0 &&&& \\ 
&& 3 && 1&&&&&&&& \\ 
&&& 2 && 1&&&&&&& \\
&&&& 2&&&&&&&&\\ 
&&&&& 1&&&&&&&\\ 
\end{array}
\end{equation*} 
\end{figure}
 Then, Figure \ref{sundaram-with-bd} shows an ice model with boundary arrows. Figure \ref{sundaram-with-GT} shows the ice model with the boundary conditions and the Gelfand-Tsetlin-Type entries filled in. Finally, Figure \ref{sundaram-filled} shows the fully filled-in ice model.

\begin{figure}[H]\centering 
\minipage[b]{0.3\textwidth}\centering
\begin{tikzpicture}
%lattice
\foreach \x in {0,1,2}{
	\draw[thick] (\x, 0) -- (\x, 6);}
\foreach \y in {1,2,3,4,5}{
	\draw[thick] (-1,\y) -- (2.5,\y);}
%bends
\foreach \top in {4,2}{
	\draw[thick] (2.5, \top) to [out=0, in=90] (3, \top - 0.5) to [out=-90, in=0](2.5, \top - 1);
	\fill (3, \top - 0.5) circle (2pt);}
%down arrows
\foreach \x/\y in {0/1,1/1,2/1}{
	\draw[arrowMe=stealth] (\x, \y) -- (\x, \y-1);}
%right arrows
\foreach \x/\y in {-1/1,-1/2,-1/3,-1/4,-1/5,2/5}{
	\draw[arrowMe=stealth] (\x, \y) -- (\x+1, \y);}
% row labels
	\node [label=right:$2$] at (3, 3) {};
	\node [label=right:$\bar{1}$] at (3, 2) {};
	\node [label=right:$1$] at (3, 1) {};
	\node [label=right:$\bar{2}$] at (3, 4) {};
	\node [label=right:$0$] at (3, 5) {};
% column labels
	\node [label=right:$3$] at (-0.5, 6.5) {};
	\node [label=right:$2$] at (0.5, 6.5) {};
	\node [label=right:$1$] at (1.5, 6.5) {};
\end{tikzpicture}\caption{}\label{sundaram-with-bd}
\endminipage
\minipage[b]{.3\textwidth}\centering
\begin{tikzpicture}
%lattice
\foreach \x in {0,1,2}{
	\draw[thick] (\x, 0) -- (\x, 6);}
\foreach \y in {1,2,3,4,5}{
	\draw[thick] (-1,\y) -- (2.5,\y);}
%bents
\foreach \top in {4,2}{
	\draw[thick] (2.5, \top) to [out=0, in=90] (3, \top - 0.5) to [out=-90, in=0](2.5, \top - 1);
	\fill (3, \top - 0.5) circle (2pt);}
%down arrows
\foreach \x/\y in {0/1,1/1,2/1,2/6,1/5,0/4,0/3,2/3,0/2,1/2}{
	\draw[arrowMe=stealth] (\x, \y) -- (\x, \y-1);}
%up arrows
\foreach \x/\y in {0/5,1/5,0/4,2/4,1/3,2/3,1/2,2/1} {
    \draw[arrowMe=stealth] (\x, \y) --(\x, \y+1);}
%right arrows
\foreach \x/\y in {-1/1,-1/2,-1/3,-1/4,-1/5,2/5}{
	\draw[arrowMe=stealth] (\x, \y) -- (\x+1, \y);}
% row labels
	\node [label=right:$2$] at (3, 3) {};
	\node [label=right:$\bar{1}$] at (3, 2) {};
	\node [label=right:$1$] at (3, 1) {};
	\node [label=right:$\bar{2}$] at (3, 4) {};
	\node [label=right:$0$] at (3, 5) {};
% column labels
	\node [label=right:$3$] at (-0.5, 6.5) {};
	\node [label=right:$2$] at (0.5, 6.5) {};
	\node [label=right:$1$] at (1.5, 6.5) {};
\end{tikzpicture}\caption{}\label{sundaram-with-GT}
\endminipage
\minipage[b]{.3\textwidth}\centering
\begin{tikzpicture}
%lattice
\foreach \x in {0,1,2}{
	\draw[thick] (\x, 0) -- (\x, 6);}
\foreach \y in {1,2,3,4,5}{
	\draw[thick] (-1,\y) -- (2.5,\y);}
%bents
\foreach \top in {4,2}{
	\draw[thick] (2.5, \top) to [out=0, in=90] (3, \top - 0.5) to [out=-90, in=0](2.5, \top - 1);
	\fill (3, \top - 0.5) circle (2pt);}
%down arrows
\foreach \x/\y in {0/1,1/1,2/1,2/6,1/5,0/4,0/3,2/3,0/2,1/2}{
	\draw[arrowMe=stealth] (\x, \y) -- (\x, \y-1);}
%up arrows
\foreach \x/\y in {0/5,1/5,0/4,2/4,1/3,2/3,1/2,2/1} {
    \draw[arrowMe=stealth] (\x, \y) --(\x, \y+1);}
%right arrows
\foreach \x/\y in {-1/1,-1/2,-1/3,-1/4,-1/5,2/5,0/5,1/4,0/3,1/3,0/2,0/1,1/1}{
	\draw[arrowMe=stealth] (\x, \y) -- (\x+1, \y);}
%left arrows
\foreach \x/\y in {2/5,1/4,2/2} {
    \draw[arrowMe=stealth] (\x, \y) --(\x-1, \y);}
%A-Type bents
\foreach \x/\y in {2.5/4,2.5/2}{
	\draw[arrowMe=stealth] (\x, \y) to [out=0, in=90](\x + 0.5, \y - 0.5);
	\draw[arrowMe=stealth] (\x+0.5, \y - 0.5) to [out=-90, in=0](\x, \y - 1);}
% row labels
	\node [label=right:$2$] at (3, 3) {};
	\node [label=right:$\bar{1}$] at (3, 2) {};
	\node [label=right:$1$] at (3, 1) {};
	\node [label=right:$\bar{2}$] at (3, 4) {};
	\node [label=right:$0$] at (3, 5) {};
% column labels
	\node [label=right:$3$] at (-0.5, 6.5) {};
	\node [label=right:$2$] at (0.5, 6.5) {};
	\node [label=right:$1$] at (1.5, 6.5) {};
\end{tikzpicture}\caption{}\label{sundaram-filled}
\endminipage
\end{figure}
\end{ex}

\begin{lemma}
If $a_{2k+1,n-k+1}=0$, then the loop on the right boundary indexed by $n-k+1$ and $\overline{n-k+1}$ goes counterclockwise. If not, then the loop goes clockwise. Further, in row $0$, the arrow points out.
\end{lemma}
\begin{proof}
We begin by proving the latter statement. Suppose instead of loops on the right side of our ice model, we create a new model where we include a column labeled 0. Let $(x, y)$ denote the vertex of the ice model in column $x$ and row $y$. First, because of our convention that the first row of the Gelfand-Tsetlin type pattern ends in $0$, the the vertical arrow above the vertex $(0, 1)$ points up. Because of the second constraint on strict Gelfand-Tsetlin type patterns, the vertical arrow below vertex $(0, 1)$ points out. Thus, vertex $(0, 1)$ has two edges pointed out, and the other two arrows must point towards the vertex. Hence, in row 0, the far right arrow points out.\\
\par We now prove the first statement of the lemma. We make use of Lemma 2 in Section 4 from a paper by Brubaker, Bump, and Friedberg \cite{brubaker_schur_2011}. From this lemma, we get that, excluding the first row, the rightmost arrows (in the $0$ column) must alternate in our model. Now we need to determine the edge above the $(2(n-k+1)+1,0)$ vertex. Because of constraint $2$ on strict Gelfand-Tsetlin type patterns, the edge below this vertex points away. By the edge to the right of the vertex points in. Thus, if $a_{2k+1,n-k+1}=0$ in our GT pattern, then the arrow pointing above the $(2(n-k+1)+1,0)$ vertex points up, and hence, the edge to the left of the $(2(n-k+1)+1,0)$ vertex points towards it. Further, the edge to the left of the $(2(n-k+1),0)$ vertex points away from the vertex. Hence, our loop will be counterclockwise. The argument uses the same reasoning if $a_{2k+1,n-k+1}\neq0$.
\end{proof}
We consider the state of a vertex to be either $\{NE,NS,NW,EW,$ $SE,SW\}$, where the state is determined by which arrows pointed in towards the vertex. In particular, if the top edge points in and the right edge points in at a vertex, then we label it with $NE$.\\ 
\begin{lemma}
No vertex in the top row has a state of $NE$ if and only if the fourth condition of the Gelfand-Tsetlin Patterns is satisfied.
\end{lemma}
\begin{proof}
First, suppose, toward a contradiction, that there exists a vertex with a $NE$ state in the top row. In particular, this cannot be the top left vertex because the top left vertex always points from the West. Further, let us assume this is the left most $NE$ vertex. Thus, the next vertex over must have top edge pointing away from the vertex. Now, if this is an $EW$ state, then we have a contradiction, since the arrow pointing up on the top edge corresponds to an element $\lambda_{k}$ of our initial partition. Further, condition 4 on our GT patterns imply that either $\lambda_{k}$ or $\lambda_{k}-1$ must appear in our second row of the GT pattern. Thus, the vertex next to $NE$ must be a $SE$ vertex. \\
\par Suppose the next vertex is an $EW$ vertex. Again we have a contradiction, because we have $\lambda_{k}+1$ and $\lambda_{k}$ appearing in our initial partition. Thus, at least two of $\lambda_{k}+1$, $\lambda_{k}$, and $\lambda_{k}-1$ must appear in the second row, but this is not the case. Thus, we must have another $SE$ vertex. Thus suppose we have $j$ $SE$ vertices to the left of the $NE$ vertex, and because of the construction, eventually we must get a vertex with $EW$ vertex as the left side of our model points in. In this case, we have j+1 consecutive arrows pointing out on the top edges. However, we only have j arrows pointing up from the vertices of the second row. Thus, there exists an element in our GT pattern such that $a_{1,k}-a_{2,k}\nleq 1$. Then this vertex must have an $EW$ configuration and the vertex to its right must be $EW$.\\ 
\par For the other direction, suppose, toward a contradiction, that the fourth rule is violated. Let $\alpha_1,...,\alpha_n$ denote the spots where the arrows point up in a row and $\beta_1,...,\beta_m$ denote the spots where the arrows point up in the row below it. Then there exists some $\alpha_i$ and $\beta_i$ such that $\alpha_i>\beta_{i}+1$. So we have the situation where the $EW$ and $NE$ configurations are forced and we obtain the desired contradiction.
\end{proof}

\section{Deformation Formula for the Sundaram Case}\label{DefFormSundarm}
We are now going to use the results from Section \ref{SundaramSection} in order to prove a Tokuyama-like formula for $SO(2n+1,\mathbb{C})$ coming from our new ice models. 
\subsection{Tokuyama-like formula for $SO(2n+1,\mathbb{C})$}\label{MainTheorem}
In \cite{gray_metaplectic_2017}, Gray creates ice models for the $Sp(2n,\mathbb{C}$, which look like our ice models without the top row. We use the weights as he defines them which we show in Figure \ref{BWDeltGamIce}. In \cite{ivanov_thesis} Ivanov uses ice models with an alternating boundary on the left, which are in bijection with the Gray Ice Models.\\
\par We now use$\lambda= (\lambda_{n}\geq\lambda_{n-1},\geq\dots,\geq \lambda_{1}>0)$ and $\rho=(n,n-1,\dots,2,1)$. Further, let us create our ice models for $(\lambda+\rho)$. Let our Boltzmann weights be those defined in Figures \ref{BoltzmannWeights} and \ref{BendWeights}. Next, let $\mathcal{Z}(\lambda)$ be the partition function defined on our ice states with top row $(\lambda+\rho)$. Let $C^{*}=\mathbf{z}^{-\rho}\prod_{i=1}^{n}(1+tz_{i}^{2})\prod_{i<j}(1+tz_{i}z_{j})(1+z_{i}z_{j}^{-1})$, where $\mathbf{z}^{-\rho}=z_{1}^{-n}z_{2}^{-n+1}\dots z_{n}^{-1}$ (this formula is a deformation of Cartan type C and can be found in \cite{brubaker_six-vertex_2015}). Let $s_{\lambda}^{so}$ be the character of the irreducible representation of $SO(2n+1)$ indexed by $\lambda$. Then we obtain our main result:\\ 
\begin{thm}\label{CdeformationtimesBcharacter}
Using the Boltzmann weights provided below, the partition function on the ice models is a deformation of the character:
\begin{equation*}
\mathcal{Z}(\lambda)=C^{*}s_{\lambda}^{so}
\end{equation*}
\end{thm}

\begin{figure}[!ht]
\begin{center}

\begin{tabular}
{|c|c|c|c|c|c|}
\hline

%&&&&&\\
$\begin{tikzpicture}
\node [label=left:$ j$] at (0.5,1) {};
\node [label=right:$ $] at (1.5,1) {};
\node [label=above:$ $] at (1,1.5) {};
\node [label=below:$ $] at (1,0.5) {};
\draw [>->] (0.5,1) -- (1.5,1);
\draw [>->] (1,1.5) -- (1,0.5);
\end{tikzpicture}$
&
$\begin{tikzpicture}
\node [label=left:$j$] at (0.5,1) {};
\node [label=right:$ $] at (1.5,1) {};
\node [label=above:$ $] at (1,1.5) {};
\node [label=below:$ $] at (1,0.5) {};
\draw [<-<] (0.5,1) -- (1.5,1);
\draw [<-<] (1,1.5) -- (1,0.5);
\end{tikzpicture}$
&
$\begin{tikzpicture}
\node [label=left:$j$] at (0.5,1) {};
\node [label=right:$ $] at (1.5,1) {};
\node [label=above:$ $] at (1,1.5) {};
\node [label=below:$ $] at (1,0.5) {};
\draw [>->] (0.5,1) -- (1.5,1);
\draw [<-<] (1,1.5) -- (1,0.5);
\end{tikzpicture}$
&
$\begin{tikzpicture}
\node [label=left:$j$] at (0.5,1) {};
\node [label=right:$ $] at (1.5,1) {};
\node [label=above:$ $] at (1,1.5) {};
\node [label=below:$ $] at (1,0.5) {};
\draw [<-<] (0.5,1) -- (1.5,1);
\draw [>->] (1,1.5) -- (1,0.5);
\end{tikzpicture}$
&
$\begin{tikzpicture}
\node [label=left:$j$] at (0.5,1) {};
\node [label=right:$ $] at (1.5,1) {};
\node [label=above:$ $] at (1,1.5) {};
\node [label=below:$ $] at (1,0.5) {};
\draw [<->] (0.5,1) -- (1.5,1);
\draw [>-<] (1,1.5) -- (1,0.5);
\end{tikzpicture}$
&
$\begin{tikzpicture}
\node [label=left:$j$] at (0.5,1) {};
\node [label=right:$ $] at (1.5,1) {};
\node [label=above:$ $] at (1,1.5) {};
\node [label=below:$ $] at (1,0.5) {};
\draw [>-<] (0.5,1) -- (1.5,1);
\draw [<->] (1,1.5) -- (1,0.5);
\end{tikzpicture}$

$\Delta$ Ice
\\
%&&&&&\\ 
\hline

$1$&$tz_i$&$1$&$z_i$&$z_i(t+1)$&$1$\\ \hline

\hline

%&&&&&\\
$\begin{tikzpicture}
\node [label=left:$j$] at (0.5,1) {};
\node [label=right:$ $] at (1.5,1) {};
\node [label=above:$ $] at (1,1.5) {};
\node [label=below:$ $] at (1,0.5) {};
\draw [>->] (0.5,1) -- (1.5,1);
\draw [>->] (1,1.5) -- (1,0.5);
\end{tikzpicture}$
&
$\begin{tikzpicture}
\node [label=left:$j$] at (0.5,1) {};
\node [label=right:$ $] at (1.5,1) {};
\node [label=above:$ $] at (1,1.5) {};
\node [label=below:$ $] at (1,0.5) {};
\draw [<-<] (0.5,1) -- (1.5,1);
\draw [<-<] (1,1.5) -- (1,0.5);
\end{tikzpicture}$
&
$\begin{tikzpicture}
\node [label=left:$j$] at (0.5,1) {};
\node [label=right:$ $] at (1.5,1) {};
\node [label=above:$ $] at (1,1.5) {};
\node [label=below:$ $] at (1,0.5) {};
\draw [>->] (0.5,1) -- (1.5,1);
\draw [<-<] (1,1.5) -- (1,0.5);
\end{tikzpicture}$
&
$\begin{tikzpicture}
\node [label=left:$j$] at (0.5,1) {};
\node [label=right:$ $] at (1.5,1) {};
\node [label=above:$ $] at (1,1.5) {};
\node [label=below:$ $] at (1,0.5) {};
\draw [<-<] (0.5,1) -- (1.5,1);
\draw [>->] (1,1.5) -- (1,0.5);
\end{tikzpicture}$
&
$\begin{tikzpicture}
\node [label=left:$j$] at (0.5,1) {};
\node [label=right:$ $] at (1.5,1) {};
\node [label=above:$ $] at (1,1.5) {};
\node [label=below:$ $] at (1,0.5) {};
\draw [<->] (0.5,1) -- (1.5,1);
\draw [>-<] (1,1.5) -- (1,0.5);
\end{tikzpicture}$
&
$\begin{tikzpicture}
\node [label=left:$j$] at (0.5,1) {};
\node [label=right:$ $] at (1.5,1) {};
\node [label=above:$ $] at (1,1.5) {};
\node [label=below:$ $] at (1,0.5) {};
\draw [>-<] (0.5,1) -- (1.5,1);
\draw [<->] (1,1.5) -- (1,0.5);
\end{tikzpicture}$

$\Gamma$ Ice
\\
%&&&&&\\ 
\hline

$1$&$z_i^{-1}$&$t$&$z_i^{-1}$&$z_i^{-1}(t+1)$&$1$\\ \hline
\end{tabular}
\caption{Boltzmann weights for $\Delta$ and $\Gamma$ Sundaram Ice }
\label{BWDeltGamIce}
\end{center}

\begin{center}
\begin{tikzpicture}[scale=1]
\node[anchor = west, inner sep = 3pt] at (4.5, 1) {$z_i^{-1}$};
\node[anchor = west, inner sep = 3pt]  at (8.5,1) {$t$};

\draw[thick, arrowMe=stealth] (4.5,1) to [out=-90, in=0] (4, 0.5);
\draw[thick, arrowMe=stealth](4, 1.5) to [out=0, in=90](4.5,1);
\draw[thick, arrowMe=stealth] (8.5,1) to [out=90, in=0] (8,1.5);
\draw[thick, arrowMe=stealth] (8,.5) to [out=0, in=-90] (8.5,1);

\foreach \nd in {4.5, 8.5}{
	\fill (\nd, 1) circle (2pt);}

\end{tikzpicture}
\\
\caption{Boltzmann Weights for Sundaram Bends}
\label{BendWeights}
\end{center}
\end{figure}

\begin{proof}
Using Ivanov and Gray, for top row $(\lambda+\rho)$ in the U-turn ice models, we have the partition function $\mathcal{Z}(\lambda)=C^{*}s_{\lambda}^{sp}$, where $s_{\lambda}^{sp}$ is the character of the irreducible representation of $Sp(2n)$ indexed by $\lambda$. Now, in our ice models, every top row has a product equal to 1. Thus, eliminating this top row, we have an ice model exactly like those in Gray, where the top row is exactly the second row of our GT pattern. Thus, let $(\lambda^{'}+\rho)$ be the second row of our GT pattern. Note that $|(\lambda +\rho)|=|(\lambda^{'}+\rho)|$ by construction of our ice models. Also note that if $\lambda_{1}>\lambda_{1}^{'}$, the difference is exactly 1. Also, this implies every vertex in the first column of our ice models with top row $(\lambda^{'}+\rho)$ is NW, which is a weight of 1. Thus, for any given top row $(\lambda^{'}+\rho)$, we get $\mathcal{Z}(\lambda^{'})=C^{*}s_{\lambda^{'}}^{sp}$. Thus, if we consider all possible allowed second rows of our GT patterns, we get the following:
\begin{equation*}
    \mathcal{Z}(\lambda)=\sum_{\mu\subseteq\lambda}C^{*}s_{\mu}^{sp}
\end{equation*}
where $\mu\subseteq\lambda$ if it is an allowed second row in our GT model. However, as a result of Sundaram and our Equation \ref{HorizontalStripEquation}, 
\begin{equation*}
    \sum_{\mu\subseteq\lambda}s_{\mu}^{sp}=s_{\lambda}^{so}
\end{equation*}
Thus, in particular, we get $\mathcal{Z}(\lambda)=C^{*}s_{\lambda}^{so}$
\end{proof}
\par Now we are interested in a type B character deformation times a type B character. Through an adjustment of weights from Gray, we get the following.\\
\begin{cor}[Tokuyama-like Formula for $SO(2n+1,\mathbb{C})$]\label{MainCor}
There exists a set of Boltzmann weights for ice models in the Sundaram case such that the partition function defined from our ice models is precisely:
\begin{equation}
    \mathcal{Z}(\lambda)=\prod_{i=1}^{n}(1+tz_{i})\prod_{i<j}(1+tz_{i}z_{j})(1+z_{i}z_{j}^{-1})s_{\lambda}^{so}
\end{equation}
\end{cor}
\begin{proof}
Consider Theorem \ref{CdeformationtimesBcharacter}. Using the weights in Figures \ref{BoltzmannWeights} and \ref{BendWeights}, we make the following modification to the bend weights: let $A_{i}$ and $B_{i}$ be the clockwise and counterclockwise oriented U-turn boundary connecting row $i$ and $\bar{i}$, respectively. We note that either $A_{i}$ or $B_{i}$ occurs in every ice state. Thus, let us multiply our weights for $A_{i}$ and $B_{i}$ by $ z_{i}^{-n+i-1}\dfrac{(1+tz_{i})}{(1+tz_{i}^{2})}$. Thus, this means for a given ice state, we are multiplying the weight of the model by $\mathbf{z}^{-\rho}\dfrac{\prod_{i=1}^{n}(1+tz_{i})}{\prod_{i=1}^{n}(1+tz_{i}^{2})}$. Since this happens for all our ice states, we are simply multiplying our partition function that we got from Theorem \ref{CdeformationtimesBcharacter}, by $\mathbf{z}^{-\rho}\dfrac{\prod_{i=1}^{n}(1+tz_{i})}{\prod_{i=1}^{n}(1+tz_{i}^{2})}$, giving the result.  
\end{proof}
In particular, this is a type B deformation as described in \cite{brubaker_six-vertex_2015}, which was our main goal.
\section{Koike-Terada Tableaux}\label{KTSection}
We now consider tableaux for $(SO(2n+1,\mathbb{C})$ that come from different branching rules due to Koike-Terada \cite{koike_young_1990}. We create GT patterns in bijection with tableaux, create strict GT patterns in bijection with shifted tableaux, and then form new ice models in bijection with our strict GT patterns.
\subsection{Construction}
\par Another set of tableaux rules for the $SO(2n+1)$ group is defined by Koike and Terada in \cite{koike_young_1990}.\\  Given a partition $\lambda = (\lambda_n \geq \lambda_{n - 1} \geq \cdots \geq \lambda_1 \geq 0)$, we fill a standard Young tableau of shape $\lambda$ with the alphabet $\{1 < \bar{1} < \barr{1} < 2 < \bar{2} < \barr{2} \cdots n < \bar{n} < \barr{n}\}$. Let $T_{i, j}$ be the entry of the tableau in the $i$-th row and the $j$-th column. The filling of the tableau must satisfy the following conditions:
\begin{enumerate}
    \item $k$ can only appear in $T_{k, 1}$
    \item Rows are weakly increasing
    \item Columns are strictly increasing
    \item $T_{i, j} \geq i$
\end{enumerate}
Given a tableau filled according to these rules, the weight of the tableau is consequently defined.\\
\begin{defn}
Let $T$ be a tableau of shape $\lambda$. The \textbf{weight} of $T$ is the vector of integers $wt(T) = (d_1, d_2, \dots , d_n)$, where $d_i = w(\sbar i) - w(\barr i)$ denotes the difference between the number of $\sbar i$'s and $\barr i$'s in $T$.
\end{defn}
We let $\mathbf{z}^{wt(T)} = z_1^{d_1}z_2^{d_2}\cdots z_n^{d_n}$. Then, summing over all tableaux of shape $\lambda$ yields the character formula.

\subsection{Young-Diagrammatic Description of Branching Rules}
\par Koike and Terada gave a description of the restriction rules of representations of $SO(2n+1, \CC)$ using these tableaux. Instead of using the Gelfand-Tsetlin bases of representation spaces as in Zhelobenko \cite{zhelobenko_classical_1962}, the tableaux rules given above directly determine the weights and multiplicities in an irreducible representation of the special orthogonal group. \\
Let $\lambda$ be a partition. The length of $\lambda$ is the number of non-zero terms and is denoted by $l(\lambda)$. A subpartition $\lambda \supset \mu = (\mu_n, \mu_{n-1}, \dots, \mu_n-m + 1)$ is a partition of length $m\le n$ such that $\mu_i\le \lambda_i$, for all $n-m\le i\le n$. In other words, the Young diagram of $\lambda$ contains the diagram of $\mu$. The skew diagram is the set-theoretic difference $\theta=\lambda-\mu$. \cite{macdonald_symmetric_2015}\\ 
Under the restriction rule $SO(2n+1)\downarrow SO(2n-1) \times GL(1)$, the character of the irreducible representation parametrized by $\lambda$ is related to the tableaux as follows:
\begin{multline}\label{KT-prop4.2}
\chi_{\lambda_{SO(2n+1)}}\downarrow SO(2n-2k+1)\times \overbrace{GL(1)\times \cdots \times GL(1)}^{k\text{ times}}= \sum_{\substack{\mu \subset \lambda, l(\mu)\le n-k\\T_{\lambda-\mu}}}\chi_{\mu_{SO(2n-2k+1)}}\cdot z^{wt(T)},
\end{multline}
where the summation runs over all possible fillings of skew tableaux of $\lambda-\mu$. \\
In lieu of proving Equation \ref{KT-prop4.2}, which can be done inductively, we look at the case of $k=1$, i.e. branching down one level. Each skew tableau gives rise to a monomial in $z_{n}$, hence the character of a representation of $GL(1)$. The tableaux of shape $\lambda$ thus illustrate how a polynomial representation of $SO(2n+1)$ decomposes as a direct sum of irreducible representations of its subgroups upon restriction to $SO(2n-1)\times GL(1)$. Although this description of restriction rules is independent of Gelfand-Tsetlin bases, we develop an analogue of such patterns in the next section.

\subsection{Gelfand-Tsetlin-Type Patterns}\label{KT-GT}
We developed Gelfand-Tsetlin-type patterns in bijection with the Koike-Terada tableaux. Let $a_{i, j}$ be the entry of the pattern in the $i$-th row and the $j$-th column. Given a partition  $\lambda = (\lambda_n \geq \lambda_{n - 1} \geq \cdots \geq \lambda_1 \geq 0)$, we create a pattern with top row $\lambda$ satisfying the following rules:
\begin{enumerate}
    \item The pattern has $3n$ rows. We will label these rows $1, \bar{1}, \barr{1}, \cdots, n, \bar{n}, \barr{n}$, starting from the bottom of the pattern.
    \item Rows $i, \bar{i},$ and $\barr{i}$ must have $i$ entries that weakly decrease across the row.
    \item Each entry $b$ must be in the interval $[a, c]$, where $a$ is the entry above and to the right of $b$, and $c$ is the entry above and to the left of $b$.
    \item Row $i$ must end in a $1$ or a $0$ (for $i \in \{1, \cdots, n\}$)
    \item Each entry in row $\barr{i}$ (for $\barr{i} \in \{\barr{1}, \cdots, \barr{n - 1}\}$) must be left-leaning.
\end{enumerate}
Rules $4$ and $5$ account for the first tableaux rule.\\
The correspondence between the Koike-Terada tableuax and Gelfand-Tsetlin-type patterns works in the same way as in the Sundaram case.
\begin{ex}
The Gelfand Tsetlin-type pattern
\begin{equation*}
\begin{array}{ccccccccccccc}
& 5 && 3&&&&&&&&&  \\
&& 4 && 2&&&&&&&& \\
&&& 2 && 1&&&&&&& \\
&&&& 2&&&&&&&&\\ 
&&&&& 1&&&&&&&\\
&&&&&& 0 &&&&&&
\end{array}
\end{equation*}
corresponds with the tableau
\begin{equation*}
\begin{ytableau}
    \sbar{1} & \barr{1} & \sbar{2} & \sbar{2} & \barr{2}\\
    \text{\small{\(2\)}} & \sbar{2} & \barr{2}
\end{ytableau}
%\begin{Young}
%$a$&&\cr
%&\cr
%\cr
%\end{Young}
\end{equation*}
\end{ex}

%------------------------------------------------------
\subsection{Shifted Tableaux and Strict Gelfand-Tsetlin-type patterns}
Just like in the previous section, we want to to define a bijection between shifted tableaux and strict Gelfand-Tsetlin-type patterns to construct corresponding ice models. We refer to a Gelfand-Tsetlin-type pattern as strict if each row is strictly increasing from right to left. Then the shifted tableaux rules are defined as follows:

\begin{enumerate}
    \item Rows are weakly increasing.
	\item Columns are weakly increasing.
	\item Diagonals are strictly increasing.
	\item The first entry in row $k$ is $k$, $\overline{k}$ or $\overline{\overline{k}}$.
\end{enumerate}

We note that violation of the fifth rule would result in the following configuration in the Gelfand-Tsetlin-type pattern (thus contradicting its strictness):
\begin{equation*}
\begin{array}{ccccccccccccc}
c && 0 && 0&&&&&&&&\\
& c && 0&&&&&&&&& \\
&& c && 0&&&&&&&& \\
&&& c && 0&&&&&&& \\
&&&& c&&&&&&&&\\ 
&&&&& b&&&&&&&\\
&&&&&& a &&&&&&
\end{array}
\end{equation*}

\subsection{Ice Model}\label{KTICEsection}
The next step in the characterization of Koike-Terada tableaux is to construct an ice model in bijection with strict patterns. \\
\par We define our ice model to be a grid with $3n$ horizontal and $\lambda_n$ vertical lines. The horizontal lines are labeled from $1$ to $\overline{\overline{n}}$ starting from the bottom and the vertical lines are labeled $1$ to $\lambda_{n}$ starting from the right. The shape and boundary conditions are given as follows:

\par We allow the three possible "bends" on the right boundary of each model connecting the rows labeled $\overline{\overline{k}}$ and $\overline{k}$ for each $k\in\{1,\cdots, n\}$: \\
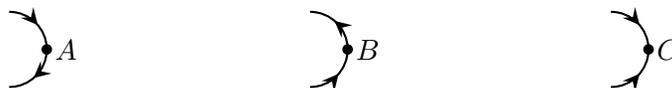
\begin{center}
\begin{tikzpicture}[scale=1]
\node[anchor = west, inner sep = 3pt] at (4.5, 1) {$A$};
\node[anchor = west, inner sep = 3pt]  at (8.5,1) {$B$};
\node[anchor = west, inner sep = 3pt]  at (12.5,1) {$C$};
\draw[thick, arrowMe=stealth] (4.5,1) to [out=-90, in=0] (4, 0.5);
\draw[thick, arrowMe=stealth](4, 1.5) to [out=0, in=90](4.5,1);
\draw[thick, arrowMe=stealth] (8.5,1) to [out=90, in=0] (8,1.5);
\draw[thick, arrowMe=stealth] (8,.5) to [out=0, in=-90] (8.5,1);
\draw[thick, arrowMe=stealth] (12,.5) to [out=0, in=-90] (12.5,1);
\draw[thick, arrowMe=stealth] (12,1.5) to [out=0, in=90] (12.5,1);
\foreach \nd in {4.5, 8.5, 12.5}{
	\fill (\nd, 1) circle (2pt);}

%old implementation
\iffalse
\node [label=right:$A$] at (4.5,.5) {};
\node [label=right:$B$] at (8.5,.5) {};
\node [label=right:$C$] at (12.5,.5) {};
\draw [very thick,*->]
	(4.5,0.5) arc (0:-90:.5);
\draw [very thick,>-]
	(4,1) arc (90:0:.5);
\draw [very thick,<-]
	(8,1) arc (90:0:.5);
\draw [very thick,*-<]
	(8.5,0.5) arc (0:-90:.5);
	\draw [very thick,>-]
	(12,1) arc (90:0:.5);
\draw [very thick,*-<]
	(12.5,0.5) arc (0:-90:.5);
\fi
\end{tikzpicture}
\end{center}

\par Note that the loop of the type \begin{tikzpicture}[scale=.5] \draw [thin,->]
	(4.5,0) arc (0:-90:.5);
\draw [thin,<-]
	(4,0.5) arc (90:0:.5); \coordinate (a) at (4.5,0);
\fill (a) circle (2pt); \end{tikzpicture} is not possible, as it would require having two consecutive $NS$ configurations in the same column, which clearly can't happen. We also allow the three possible configurations called "ties" on the right boundary of each row labeled $k\in\{1,\cdots, n\}$: \\
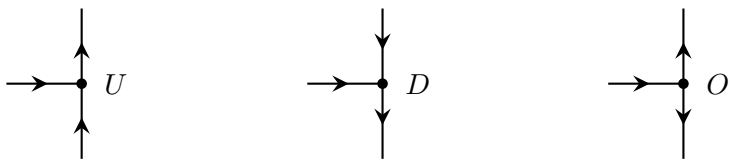
\begin{center}
    
\begin{tikzpicture}[line width=0.5mm]

\coordinate (a) at (1,2);
\fill (a) circle (2pt);
\coordinate (b) at (5,2);
\fill (b) circle (2pt);
\coordinate (c) at (9,2);
\fill (c) circle (2pt);

\node [label=right:$U$] at (1,2) {};
\node [label=right:$D$] at (5,2) {};
\node [label=right:$O$] at (9,2) {};

\draw [thick,arrowMe=stealth] (0,2) -- (1,2);
\draw [thick, arrowMe=stealth] (1,1) -- (1,2);
\draw [thick,arrowMe=stealth] (1,2) -- (1,3);

\draw [thick,arrowMe=stealth] (4,2) -- (5,2);
\draw [thick,arrowMe=stealth] (5,2) -- (5,1);
\draw [thick,arrowMe=stealth] (5,3) -- (5,2);

\draw [thick,arrowMe=stealth] (8,2) -- (9,2);
\draw [thick,arrowMe=stealth] (9,2) -- (9,1);
\draw [thick,arrowMe=stealth] (9,2) -- (9,3);

\end{tikzpicture}
\end{center}

\par Furthermore, every row labeled $k\in\{1,\cdots, n\}$ is a three-vertex model: the only vertex configurations are SW, NW, and NE.\\

\par The next three figures demonstrate boundary conditions and the full ice model for one of the patterns of $\lambda=(2,1)$.

\minipage[htb]{0.3\textwidth}\centering 
\begin{figure}[H]\centering 
\begin{tikzpicture}
%lattice
\foreach \x in {0,1}{
	\draw[thick] (\x, -1) -- (\x, 6);}
\foreach \y in {1,2,4,5}{
	\draw[thick] (-1,\y) -- (1.5,\y);}
%ties
\foreach \y in {0,3}{
	\draw[thick] (-1, \y) -- (1, \y);
	\fill (1, \y) circle (2pt);}
%bents
\foreach \top in {5, 2}{
	\draw[thick] (1.5, \top) to [out=0, in=90] (2, \top - 0.5) to [out=-90, in=0](1.5, \top - 1);
	\fill (2, \top - 0.5) circle (2pt);}
%down arrows
\foreach \x/\y in {0/0,1/0}{
	\draw[arrowMe=stealth] (\x, \y) -- (\x, \y-1);}
%right arrows
\foreach \x/\y in {-1/0,-1/1,-1/2,-1/3,-1/4,-1/5,0/0,0/3}{
	\draw[arrowMe=stealth] (\x, \y) -- (\x+1, \y);}

		% row labels
	\node [label=right:$\barr{1}$] at (2, 2) {};
	\node [label=right:$\bar{1}$] at (2, 1) {};
	\node [label=right:$1$] at (2, 0) {};
	\node [label=right:$2$] at (2, 3) {};
	\node [label=right:$\bar{2}$] at (2, 4) {};
	\node [label=right:$\barr{2}$] at (2, 5) {};

	% column labels
	\node [label=right:$2$] at (-0.5, 6.5) {};
	\node [label=right:$1$] at (0.5, 6.5) {};

\end{tikzpicture}
    \caption{Boundary}
\end{figure}
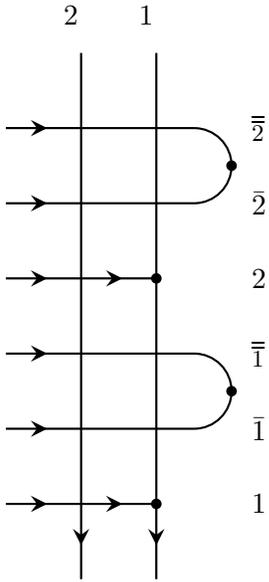
\endminipage 
\minipage[htb]{0.3\textwidth}\centering
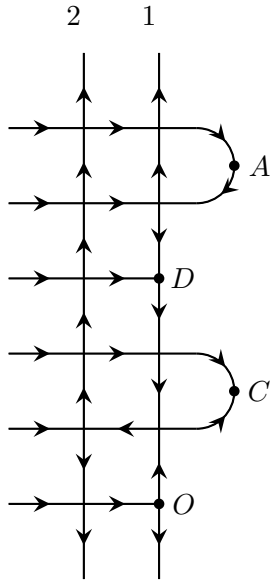
\begin{figure}[H]\centering
 \begin{tikzpicture}
%lattice
\foreach \x in {0,1}{
	\draw[thick] (\x, -1) -- (\x, 6);}
\foreach \y in {1,2,4,5}{
	\draw[thick] (-1,\y) -- (1.5,\y);}
%ties
\foreach \y in {0,3}{
	\draw[thick] (-1, \y) -- (1, \y);
	\fill (1, \y) circle (2pt);}
%bents
\foreach \top in {5, 2}{
	\draw[thick] (1.5, \top) to [out=0, in=90] (2, \top - 0.5) to [out=-90, in=0](1.5, \top - 1);
	\fill (2, \top - 0.5) circle (2pt);}
%A-Type bents (enter the top vertex)
\foreach \x/\y in {1.5/5}{
	\draw[arrowMe=stealth] (\x, \y) to [out=0, in=90](\x + 0.5, \y - 0.5);
	\draw[arrowMe=stealth] (\x+0.5, \y - 0.5) to [out=-90, in=0](\x, \y - 1);}
%C-Type bents (enter the middle vertex)
\foreach \x/\y in {2/1.5}{
	\draw[arrowMe=stealth] (\x - 0.5, \y + 0.5) to [out=0, in=90](\x, \y);
	\draw[arrowMe=stealth] (\x-0.5, \y - 0.5) to [out=0, in=-90](\x, \y);}
%up arrows
\foreach \x/\y in {0/1,0/2,0/3,0/4,0/5,1/4,1/5,1/0}{
	\draw[arrowMe=stealth] (\x, \y) -- (\x, \y+1);}
%down arrows
\foreach \x/\y in {0/1,0/0,1/0,1/2,1/3,1/4}{
	\draw[arrowMe=stealth] (\x, \y) -- (\x, \y-1);}
%right arrows
\foreach \x/\y in {-1/0,-1/1,-1/2,-1/3,-1/4,-1/5,0/0,0/3, 0/4,0/5,0/2}{
	\draw[arrowMe=stealth] (\x, \y) -- (\x+1, \y);}
	%left arrows
\foreach \x/\y in {1/1}{
	\draw[arrowMe=stealth] (\x, \y) -- (\x-1, \y);}
\node [label=left:$A$] at (2.75,4.5) {};
\node [label=left:$D$] at (1.75,3) {};
\node [label=left:$C$] at (2.75,1.5) {};
\node [label=left:$O$] at (1.75,0) {};
	
	% column labels
	\node [label=right:$2$] at (-0.5, 6.5) {};
	\node [label=right:$1$] at (0.5, 6.5) {};

\end{tikzpicture}
\caption{Ice Model}
\end{figure}
\endminipage \hfill
\minipage[htb]{.3\textwidth}\centering

\begin{figure}[H]
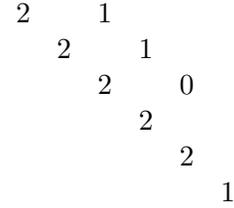
\centering
\begin{equation*}
\begin{array}{ccccccccccccc}

&&&&&&&&&&&&\\
&&&&&&&&&&&&\\
&&&&&&&&&&&&\\
&&&&&&&&&&&&\\
&&&&&&&&&&&&\\
& 2 && 1 &&&&&&&&& \\ 
&& 2 && 1&&&&&&&& \\ 
&&& 2 && 0&&&&&&& \\
&&&& 2&&&&&&&&\\ 
&&&&& 2&&&&&&&\\ 
&&&&&& 1 &&&&&& \\
&&&&&&&&&&&&\\
&&&&&&&&&&&&\\
&&&&&&&&&&&&
\end{array}
\end{equation*} 

\caption{GT pattern}\end{figure}
\endminipage\hfill 

\par We prove that such ice models are in bijection with strict Gelfand-Tsetlin-like patterns in the next two theorems. We note that the first three Gelfand-Tsetlin-type pattern rules are automatically satisfied as in previous sections. It remains to prove the bijection between the strict Gelfand-Tsetlin-type patterns and the ice, which we will do with the following theorem:\\
\begin{thm} \label{3vertexmodel}
The following are equivalent:
\begin{enumerate}
    \item Koike-Terada Gelfand-Tsetlin-type pattern rules 4 and 5 are satisfied.
    \item Each ice row labeled $k\in\{1,\cdots, n\}$ has no $NS$, $SE$, or $EW$ configurations, and tie boundary conditions are satsified.
\end{enumerate}
\end{thm}

\begin{proof} We'll show the two directions:
\par First, suppose that rules 4 and 5 are satisfied. Then clearly $NS$ isn't possible in rows labeled $k$, since all the entries in row $\barr{k-1}$ are left-leaning. Now assume, toward a contradiction, that we have a $SE$ configuration in a row labeled $k$. Then the configuration to the left of it is either $SE$ again, or $NE$, or $EW$. If it is $SE$ or $NE$, we look at the configuration to the left of it. We may now assume we reached the leftmost $SE$ or $NE$ configuration, then the state to the left must be $EW$ because of the left boundary conditions. But this is a contradiction to rule 5, since we get a non-left-leaning entry, as the up-arrow in the $EW$ configuration is not the rightmost in that row.
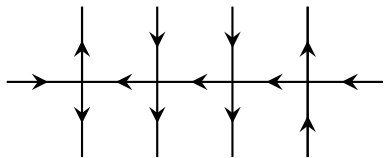
\begin{figure}[H]\centering 
\begin{tikzpicture}
%lattice
\foreach \x in {1,2,3,4}{
	\draw[thick] (\x, 0) -- (\x, 2);}
\foreach \y in {1}{
	\draw[thick] (0,\y) -- (5,\y);}
%right arrows
\foreach \x/\y in {0/1}{
	\draw[arrowMe=stealth] (\x, \y) -- (\x+1, \y);}
	%left arrows
\foreach \x/\y in {2/1, 3/1, 4/1, 5/1}{
	\draw[arrowMe=stealth] (\x, \y) -- (\x-1, \y);}
	%down arrows
\foreach \x/\y in {2/2,3/2,2/1,3/1,1/1}{
	\draw[arrowMe=stealth] (\x, \y) -- (\x, \y-1);}
	%up arrows
	\foreach \x/\y in {1/1, 4/1,4/0}{
	\draw[thick, arrowMe=stealth] (\x, \y) -- (\x, \y+1);}
\end{tikzpicture}
    \caption{SE forces EW}
\end{figure}
Now we note that rule 4 directly implies that we cannot have $EW$ in any row labeled $k\in\{1,...,n\}$. A Gelfand-Tsetlin-type pattern row ending in 0 or 1 implies that we could only have $EW$in the (imaginary) $\nth{0}$ or $\nth{1}$ row, and none of these are possible. \\ 
\par We will now show that the tie boundary conditions are also implied by rules 4 and 5. So, row $k$ must end in a $0$ or a $1$. If row $k$ ends in a $1$, then a $1$ cannot appear in row $\barr{k - 1}$ since each entry in row $\barr{k - 1}$ must be left-leaning and strictly decreasing. This would correspond to having an $EW$ in column $1$ of our ice model if we had a full column $1$, but now corresponds to the $O$-tie. If row $k$ ends in a $0$, then for similar reasoning we would have an $EW$ in the $\nth{0}$ column of our ice model, if we were to have a $\nth0$ column. As stated earlier, we know $SE$ and $NS$ cannot appear in this row, and it is clear that neither $EW$ nor $NE$ can appear directy to the left of an $EW$. Thus, the only possible fillings for the vertex in the $\nth{1}$ column would be a $SW$ or a $NW$. Having a $SW$ implies having the U tie boundary, and having a $NW$ implies the type D tie boundary.\\
\par First, suppose each non-bar row has no $NS$, $SE$, and $EW$. Then, the only vertices in the row can be $NW$, $SW$, and $NE$. Each of these types of vertices have the two vertical arrows pointing in the same direction, meaning the row below it is left leaning. Thus, rule 5 is satisfied. Next, suppose the tie boundary conditions are also satisfied. Then assume, toward a contradiction, that row $k\in\{1,\cdots,n\}$ ends in a number greater than 1. This means the only possible tie configuration for row $k$ would be D. If this is the case, all the horizontal arrows to the left of it must point right to avoid NS. However, since row $k$ in the Gelfand-Tsetlin-type pattern has one more entry than row $\barr{k-1}$, then eventually we will have a configuration where two consecutive vertical arrows in that row point away from each other, thus obtaining the desired contradiction, and rule 4 is satisfied.
\end{proof}

\section{Proctor Tableaux}\label{ProctorSection}
We now consider tableaux for $(SO(2n+1,\mathbb{C})$ that come from different branching rules due to Proctor \cite{proctor_young_1994}. We create GT patterns in bijection with tableaux. However, we show strict GT patterns do not exist naturally, and thus there is no natural correspondence with ice models.
\subsection{Construction}
We will examine one more tableaux associated to $SO(2n+1)$, as defined by Proctor.Let $\lambda=(\lambda_{n}\geq\lambda_{n-1}\geq\dots\geq\lambda_{1}\geq1)$ be a partition. We fill a standard Young Tableaux of shape $\lambda$ with alphabet and ordering $\{1<2<\dots<2n<0\}$, where 0 is infinity. We define $T_{i, j}$ to be the entry in a tableaux $T$ in row $i$ and column $j$. We fill the tableaux such that:
\begin{enumerate}
	\item Rows are weakly increasing.
	\item Columns are strictly increasing.
	\item Follows the $2c$ orthogonal condition
	\item Follows the $2m$ protection condition 
\end{enumerate}
\medskip
\begin{defn}
The \textbf{$2c$ orthogonal condition} is satisfied if for every $c \leq n$, there are less than or equal to $2c$ entries that are less than or equal to $2c$ in the first two columns, not including a 0 entry.
\end{defn}
\medskip
\begin{defn}
The \textbf{$2m$ protection condition} is satisfied if for every $m \leq n$:
\begin{enumerate}
    \item If an entry in the first column is equal to $2m-1$, then define i to be equal to the row number of said entry.
    \item Then, define $k$ to be equal to $2m-i$.
    \item If there is an entry in the $kth$ row equal to $2m$, let h be equal to the column number of the leftmost $2m$ entry.
    \item Then, for any $j$ such that $2 \leq j \leq h$, $T_{k, j}$ must be equal to $2m-1$. Additionally, $T_{k-1, h}$ must equal to 2m-1.
    \item If the suppositions do not hold for $m$, then the condition is trivially satisfied.
\end{enumerate}
\end{defn}
\par For example, if our tableaux looked like the following:
\begin{equation*}
\young(11x5,3x4,5)
\end{equation*}
Then to satisfy the $2m$ protection condition, every entry where there is an x (i.e. $T_{1,3}$ and $T_{2,2}$) must equal 3.
If our tableaux looked like the following:
\begin{equation*}
\young(1135,x34,5)
\end{equation*}
Then the entry $T_{2, 1}$ must be greater than or equal to 3 in order to satisfy the 2c orthogonal rule (and, in this case, it must be 3 so that the 2nd row is weakly increasing). If $T_{2,1}$ were equal to 2, then when $c=1$, there would be more than $2c$ entries in the first two columns that are less than or equal to $2c$. 
The following tableaux satisfies all of the rules of a valid Proctor tableaux:
\begin{equation*}
\young(1135,334,5)
\end{equation*}

\subsection{Branching Rules}

We have that the indexing set for these two point-wise equal sets of positive tensor orthogonal characters coincide for all n for $SO(2n+1)$ and $O_{2n+1}$. The rules for Proctor's tableaux are a consequence of the branching restriction $O_{2n+1}\downarrow O_{2n-1} \otimes O_{2}$ at each step, where the restriction to $O_{2n-1}$ gives us the Protection Condition and the restriction to $O_{2}$ gives us the Orthogonal Condition. Proctor gives the following proposition. \\ 

\begin{prop}[Proctor 10.3] Let $G_N = O_N$, and let $\lambda$ respectively be an N-orthogonal partition. Then the representation $G_N(\lambda)$ restricts to $\oplus G_{N-2}(v) \otimes G_2(\lambda - v)$, where the sum is over all $v \subseteq \lambda$ which are possible indexing partitions for representations of $G_{N-2}$ such that $\lambda - v$ does not have more than 2 boxes in a column.

\end{prop}

\subsection{Proctor Tableaux and Gelfand-Tsetlin-Type Patterns}\label{Proctor-GT}

To create Gelfand-Tsetlin-type patterns to correspond with the odd special orthogonal tableaux rules given by Proctor, we must take into consideration the effect of the $2c$ orthogonal condition and the $2m$ protection condition. We begin with our top row entries being equal to $\lambda$. The top $n$ rows will have $n$ entries. The n+1st row will have $n-1$ entries, the $n+2$nd row will have $n-2$ entries, etc, and the bottom row will have 1 entry.\\ 

\par To biject with the $2m$ protection condition, we use the following rules in our patterns. \begin{itemize}
    \item In order to check the following rules, all of the rows in our Gelfand-Tsetlin-type patterns must have $n$ entries. We add 0s to the ends of our rows that have less than $n$ entries such that each row now has $n$ entries (including the added 0's). Additionally, we add a "0th row" at the bottom of our Gelfand-Tsetlin-type pattern that is made of n 0's. 
    \item If there is a non-left-leaning 0 entry in an even row, we define $m$ and $i$ so that $a_{2m-2, i}=0$. 
    \item Then, define $j = 2m-i$.
    \item We check to see if $a_{2m, j} \geq a_{2m-1, j}$. If it is, we check the following:
    \begin{itemize}
        \item $a_{2m-1, j-1} \geq a_{2m-2, j-1}$.
        \item $a_{2m-1,j-1} \geq a_{2m, j}$
        \item $a_{2m-2, j} \leq 1$ 
    \end{itemize}
\end{itemize}

\par To biject with the $2c$ orthogonal condition, we have a restriction on the entries in an even row, $2c$. For entries in the $2c$ row, let the orthogonal sum be the sum of the entries such that for $1 \leq k \leq n$, $a_{2c, k} = a_{2c, k}$ for $a_{2c, k} \leq 2$ and $a_{2c, k} = 2$ for $a_{2c, k} > 2$. This sum must be less than or equal to $2c$.

\subsection{Proctor Tableaux and Ice}\label{ProctorICE}

Using the above restrictions on Gelfand-Tsetlin-type patterns, we show that corresponding ice models do not follow naturally. In the Sundaram and Koike-Terada bijections to ice, only Gelfand-Tsetlin-type patterns that are strictly decreasing across rows have corresponding ice models. Non-strict Gelfand-Tsetlin-type patterns do not have corresponding statistical mechanical meaning in the form of ice models.\\ 
\begin{prop}
There exists no ice model in bijection with Proctor tableaux.
\end{prop}
\begin{proof}
In the case of $n=4$, we can look at the fourth row. The smallest possible strict configuration would be $\lambda_{4} = 3, 2, 1, 0$. However, the orthogonal sum of that row, according to the orthogonal restrictions on Gelfand-Tsetlin-type patterns, is $5$, which is greater than $4$, thus breaking the restriction. Therefore, in the case of $n=4$, there are no strict Gelfand-Tsetlin-type patterns, and consequently no valid ice models. 
\end{proof}

\section{Concluding Remarks}
We now have a fully functioning ice model stemming from Sundaram tableaux that produces a type B deformation times a type B character. However, we used just one of many possible tableaux rules. In particular, we know Koike-Terada tableaux rules produce valid statistical mechanical ice models. Thus, determining weights for these ice models and producing a type B Weyl character deformation would be useful. We also note that there does not seem to be any research on ice models for Cartan type D. Thus, a potential future research direction is considering the type D case.

\nocite{*}

\printbibliography
\newpage
\Addresses

\end{document}